\documentclass{article}
\usepackage[utf8]{inputenc}
\usepackage{amsmath,amssymb,amsthm,eucal,mathrsfs}
\usepackage[final]{hyperref}
\usepackage{todonotes}

\newtheorem{proposition}{Proposition}
\newtheorem{theorem}{Theorem}

\newtheorem{corollary*}{Corollary}

\theoremstyle{remark}
\newtheorem*{remark*}{Remark}
\newtheorem*{remark-notat*}{A remark on notation}

\newcommand \Conf {{\mathrm {Conf}}}
\newcommand  \fin {{\mathrm{fin}}}
\newcommand  \Mfin {\mathfrak{M}_\fin}
\newcommand\scrI{{\mathscr I}}

\newcommand  \loc {{\mathrm{loc}}}
\DeclareMathOperator{\re}{Re}
\DeclareMathOperator{\im}{Im}
\DeclareMathOperator{\sgn}{sgn}

\newcommand\blfootnote[1]{%
	\begin{NoHyper}
		\renewcommand\thefootnote{}\footnote{#1}%
		\addtocounter{footnote}{-1}%
	\end{NoHyper}
}

\begin{document}

\title{A Palm hierarchy for determinantal point processes with the confluent hypergeometric kernel, the decomposing measures in the problem of harmonic analysis on the infinite-dimensional unitary group}

\author{Alexander I. Bufetov\footnote{
		Department of mathematics and computer science of St. Petersburg State University,\\
		Steklov Mathematical Institute of the Russian Academy of Sciences, Moscow, Russia,\\
		Institute for Information Transmission Problems of the Russian Academy of Sciences, Moscow, Russia,\\
		Centre national de la recherche scientifique, France.}
	}
     
\date{}

\maketitle

\begin{abstract}
  The main result of this note is that the shift of the parameter by~1 in the parameter space of decomposing measures in the problem of harmonic analysis  on the infinite-dimensional unitary group corresponds to the taking of the reduced Palm measure at infinity for our decomposing measures. The proof proceeds by finite-dimensional approximation of our measures by orthogonal polynomial ensembles. The key remark is that the taking the reduced Palm measure commutes with the scaling limit transition from finite to infinite particle systems.\blfootnote{
  The author was supported by a grant of the Government of the Russian Federation for the state support of scientific research, carried out under the supervision of leading scientists, agreement 075-15-2021-602.
  }
\end{abstract}
\section{Introduction}

In the problem of harmonic analysis on infinite-dimensional groups, such as, for instance, the infinite-dimensional unitary group, the decomposing measure is in many cases a determinantal point process on a subset of the real line. Recall that $U(\infty)$-invariant ergodic measures on the space of infinite Hermitian matrices are naturally parametrized by countable subsets of $\mathbb{R}$ with no accumulation points except~$0$. An ergodic decomposition measure is naturally identified with the measure on the space of infinite configurations on~$\mathbb{R}\setminus\{0\}$. The analogue of the Haar measure in the infinite-dimensional case is the family of Hua---Pickrell measures. The latter are parametrized by one complex parameter~$s$. We thus arrive at a $1$-parameter family of decomposing measures on the space of configurations on~$\mathbb{R}\setminus\{0\}$.

The aim of this paper is to establish the connection between the decomposing measures corresponding to parameters~$s$ and $s+1$. We will see that the transition $s\to s+1$ has a clear probabilistic interpretation: the decomposing measure corresponding to the parameter $s+1$ is the Palm measure at $\infty$ of the decomposing measure corresponding to the parameter $s$.

The proof is based on a simple idea. The Borodin---Olshanski theorem identifies the decomposing measure corresponding to the parameter $s$ with the determinantal point process governed by the confluent hypergeometric kernel. The latter determinantal measure is the scaling limit of the orthogonal polynomial ensembles on the circle with the Jacobi weight
\begin{equation*}
	w^{(s)}(e^{i\theta})=(1-e^{i\theta})^s(1-e^{-i\theta})^{\bar s}.
\end{equation*}
For such orthogonal polynomial ensembles the shift $s \to s+1$ corresponds to the taking of the reduced Palm measure at the point~1. The key step of the proof is to show that the passage to the scaling limit commutes with the taking of Palm measures. From the technical point of view, the main r\^ole is played by Proposition~\ref{prop:norm-form}, a generalization of the well-known fact that zeros of orthogonal polynomials interlace.

This paper is a sequel to~\cite{Buf-TrMIAN}, in which the determinantal processes with the Bessel kernel are considered, corresponding to the decomposing measures of the Pickrell measures, that is, infinite-dimensional counterparts of the canonical invariant measures on the Grassmannians, in other words, measures on the space of infinite complex-valued matrices that are invariant with respect to the left and right multiplication by elements of the infinite unitary group. In the case of the Grassmannians the decomposing measure is obtained as a scaling limit of the Jacobi orthogonal polynomial ensembles on the line. The shift of the parameter $s\to s+1$ for the finite-dimensional approximations corresponds to the taking of the Palm measure at~$1$. For the orthogonal polynomial ensemble with the weight $w$ taking of the Palm measure at~$1$ corresponds to the transition from the weight $w(x)$ to the weight $(1-x)^2w(x)$.
The correlation kernels of our determinantal point process are then transformed by a rule that is preserved by the limit transition, and the desired recursion in the infinite-dimensional case follows.

In the discrete case a similar recursion is obtained in \cite{BufOlsh-Studia} for the family of determinantal point processes with the Gamma kernel. The main purpose of this paper is to give a general framework for obtaining the  recursion $s\to s+1$.

\paragraph*{Acknowledgments} I am deeply grateful to D.~Homza, A.~Klimenko, G.~Olshanski, and I. Pylayev for fruitful discussions and to an anonymous referee for useful remarks and suggestions.

\section{Statement of the main result}

\subsection{Hua---Pickrell measures}

We proceed to precise formulations. Let $U(n)$ be the group of $n\times n$ unitary matrices. The infinite-dimensional unitary group $U(\infty)$ is defined as the inductive limit
\begin{equation*}
	U(\infty)=\bigcup_{n=1}^\infty U(n)
\end{equation*}
with respect to the natural inclusions. Let $H(n)$ be the space of Hermitian $n\times n$ matrices. The space of infinite Hermitian matrices is denoted by~$H$. The operation of ``cutting out a corner'', that is, the removal of the last row and the last column of the matrix, yields a natural projection $H(n+1)\to H(n)$. The space $H$ is the projective limit of the spaces $H(n)$ with respect to this natural system of projections. The infinite-dimensional unitary group~$U(\infty)$ acts on $H$ by conjugations.

Let $s\in \mathbb{C}$, $\re s>-1/2$. The probability measure $\mu_n^{(s)}$ on $H(n)$ is defined by the formula
\begin{equation*}
	\mu_n^{(s)}=c_{n,s}\det(1+ih)^{-n-s}\cdot\det(1+ih)^{-n-\bar s}\,dh,
\end{equation*}
where $dh$ is the Lebesgue measure on $H(n)$ and $c_{n, s}$ is the normalizing constant. A breakthrough result of Hua Loo Keng and Neretin \cite{Hua,Neretin}, established by the former for the real values $s$ and by the latter in the general case of complex~$s$, states that the measures $\mu_n^{(s)}$ form a projective system. In other words, ``cutting the corner'' takes $\mu_{n+1}^{(s)}$ to $\mu_{n}^{(s)}$. Consequently, the projective limit
\begin{equation*}
	\mu^{(s)}=\varprojlim_{n\to\infty}\mu_n^{(s)}	
\end{equation*}
is well-defined. Borodin and Olshanski have given an explicit description of the ergodic decomposition of the measure $\mu^{(s)}$ and proved that the resulting decomposing measure is a determinantal point process $\mathbb{P}^{(s)}$ on the space $\Conf(\mathbb R)$ of point configurations on $\mathbb{R}$.

The main result of this paper identifies $\mathbb{P}^{(s+1)}$ with the reduced Palm measure at~$\infty$ of the measure $\mathbb{P}^{(s)}$.
Before presenting the precise statement we recall the definitions of determinantal point processes and Palm measures.

\subsection{Spaces of configurations}
\label{subsec-space-conf}

Let $E$ be a locally compact complete metric space. A configuration $X$ on $E$ is a subset of~$E$ without accumulation points. The points of $X$ are called particles. In other words, any compact subset of~$E$ contains only finitely many points of our configuration. To a configuration $X$ we assign the Radon measure
$$
\sum\limits_{x \in X} \delta_x,
$$
where the sum is taken over all particles of the configuration~$X$. Conversely, any purely atomic integer-valued Radon measure on $E$ is identified with a  configuration. Consequently, the space $\Conf(E)$ of configurations in $E$ is a closed subset of the space of all integer-valued Radon measures on $E$. This inclusion endows $\Conf(E)$ with a structure of a complete metric space, which is not locally compact.

The Borel structure on the space $\Conf(E)$ is defined as follows. For any relatively compact Borel set $B\subset E$ consider the function
$$
\#_{B}\colon \Conf(E) \to \mathbb{R}
$$
that to a configuration assigns the number of its particles that lie in $B$. The family of functions $\#_B$, where $B$ is taken over all relatively compact subsets of $E$, generates the Borel structure on $\Conf(E)$. In particular, joint distributions of the random variables $\#_{B_1}, \ldots, \#_{B_k}$ for all possible finite collections of pairwise disjoint relatively compact Borel subsets $B_1, \dots, B_k \subset E$ determine a probability measure on $\Conf(E)$ uniquely.

\subsection{Weak topology on space of probability measures on the space of configurations}

As we saw above, the space of configurations $\Conf(E)$ is naturally endowed with the structure of a complete metric space. The space $\Mfin(\Conf(E))$ of finite Borel measures on the space of configurations is therefore a complete metric space in the weak topology.

Let $\varphi\colon E\to\mathbb{R}$ be a bounded Borel function with a compact support. Define the measurable function $\#_{\varphi}\colon\Conf(E)\to\mathbb{R}$ by the formula
$$\#_{\varphi}(X)=\sum\limits_{x\in X}\varphi(x).$$
For a relatively compact Borel subset $B\subset E$ we have 
$\#_B=\#_{\chi_B}$.

Recall that the Borel $\sigma$-algebra on $\Conf(E)$ coincides with $\sigma$-algebra generated by the integer-valued random variables $\#_B$, where $B\subset E$ is taken over all relatively compact Borel subsets. Our Borel $\sigma$-algebra therefore coincides with the $\sigma$-algebra generated by all random variables $\#_{\varphi}$, where $\varphi\colon E\to\mathbb{R}$ is taken over all continuous functions with compact support. We arrive at the following
\begin{proposition}\label{prop:fin-distr-unique}
  	A Borel probability measure $\mathbb{P}\in \Mfin(\Conf(E))$ is defined by joint distributions $(\#_{\varphi_1}, \dots, \#_{\varphi_l})$ of all 
	finite collections of continuous functions $\varphi_1, \dots, \varphi_l\colon E\to\mathbb{R}$ with pairwise disjoint support.
\end{proposition}
The weak topology on $\Mfin(\Conf(E))$ can also be described using finite-dimen\-sional distributions (see~Daley---Vere-Jones \cite[Vol.~2, Theorem 11.1.VII]{DVJ}). Let $\mathbb{P}_n$, $n\in{\mathbb N}$ and $\mathbb{P}$ be Borel probability measures on the space $\Conf(E)$. Then the measures $\mathbb{P}_n$ converge weakly to $\mathbb{P}$ as $n\to \infty$ if and only if for any finite set of continuous functions $\varphi_1, \dots, \varphi_l$ with pairwise disjoint supports the joint distributions of the random variables $\#_{\varphi_1}, \dots, \#_{\varphi_l}$ with respect to the measure $\mathbb{P}_n$ converge as $n\to \infty$ to the joint distribution $\#_{\varphi_1}, \dots, \#_{\varphi_l}$  with respect to the measure $\mathbb{P}$ (the convergence of the joint distributions being understood with respect to the weak topology on the space of Borel measures on ${\mathbb R}^l$).

\subsection{Locally trace-class operators}
\label{subsec-loc-tr-cls}

Let $\mu$ be a $\sigma$-finite Borel measure on the space $E$.
The ideal of trace-class operators ${\widetilde K}\colon L_2(E,\mu)\to L_2(E,\mu)$ is denoted $\scrI_{1}(E,\mu)$ (see the detailed definition in \cite[Vol.~1]{ReedSimon}). The symbol 
$\|{\widetilde K}\|_{\scrI_1}$ denotes the $\scrI_{1}$-norm of the operator ${\widetilde K}$.
Let $\scrI_{1,\loc}(E,\mu)$ be the space of operators
\begin{equation*}
	K\colon L_2(E,\mu)\to L_2(E,\mu)
\end{equation*}
such that for any relatively compact Borel subset $B\subset E$ the operator $\chi_BK\chi_B$ is trace-class.
The space $\scrI_{1,\loc}(E,\mu)$ is metrized by a countable family of seminorms $\|\chi_BK\chi_B\|_{\scrI_1}$,
where $B$ is taken over an exhausting sequence $B_n$ of relatively compact Borel subsets.

\subsection{Determinantal point processes}

A Borel probability measure $\mathbb{P}$ on the space $\Conf(E)$ is called \emph{determinantal} if there exists an operator
$K\in\scrI_{1,\loc}(E,\mu)$ such that for any bounded measurable function $g$ such that the function $g-1$ is supported in a relatively compact subset $B$, we have 
\begin{equation}
	\label{eq1}
	\mathbb{E}_{\mathbb{P}}\Psi_g
	=\det\bigl(1+(g-1)K\chi_{B}\bigr).
\end{equation}
The Fredholm determinant in~\eqref{eq1} is well-defined, since $K\in \scrI_{1, \loc}(E,\mu)$.
The formula~\eqref{eq1} uniquely defines the measure $\mathbb{P}$.
For an arbitrary set of pairwise disjoint Borel subsets $B_1,\dotsc,B_l\subset E$
and any $z_1,\dotsc,z_l\in {\mathbb C}$ the formula~\eqref{eq1} implies
\begin{equation*}
	\mathbb{E}_{\mathbb{P}}z_1^{\#_{B_1}}\cdots z_l^{\#_{B_l}}
	=\det\biggl(1+\sum\limits_{j=1}^l(z_j-1)\chi_{B_j}K\chi_{\sqcup_i B_i}\biggr).
\end{equation*}

For further background on determinantal point process 
see e.~g.\ \cite{BorOx, HoughEtAl, Lyons,
LyonsSteif, Lytvynov, ShirTaka0, ShirTaka1, ShirTaka2, Soshnikov}.

Given an operator $K\in \scrI_{1,\loc}(E,\mu)$, the symbol $\mathbb{P}_K$ stands for the determinantal measure induced by~$K$. The measure $\mathbb{P}_K$ is uniquely defined by~$K$, but different operators may define the same measure.
A theorem by Macchi---Soshnikov and Shirai---Takahashi~\cite{Macchi, Soshnikov,ShirTaka1} states that any positive Hermitian contraction from~$\scrI_{1,\loc}(E,\mu)$ induces a determinantal point process.

A fundamental example of a determinantal measure is an orthogonal polynomial ensemble on the circle or on the line. Recall that an orthogonal polynomial ensemble of degree $n$ with weight $w$ on $\mathbb R$ is the measure given by the formula
\begin{equation}\label{eq:orth-r}
	Z^{-1}\cdot \prod_{1\le i<j\le n}(x_i-x_j)^2\cdot \prod_{i=1}^n w(x_i)\,dx_i,
\end{equation}
where $Z$ is the normalizing constant. (We assume here that the weight $w$ allows all moments up to $(n-1)$-th.) The measure~\eqref{eq:orth-r} is determinantal and induced by the projection operator onto the subspace
\begin{equation*}
	\operatorname{span}\{1,x,\dots,x^{n-1}\}\sqrt{w(x)}\subset L_2(\mathbb{R}).
\end{equation*}
The kernel of our projection is the $n$-th Christoffel---Darboux kernel of the system of orthogonal polynomials with weight $w$. 

An orthogonal polynomial ensemble on the unit circle is defined in a similar way. Namely, we consider the measure 
\begin{equation}\label{eq:circ-ensemble}
	Z^{-1}\cdot \prod_{1\le k<l\le n} \bigl|e^{i\theta_k}-e^{i\theta_l}\bigr|^2 \cdot \prod_{k=1}^n w(e^{i\theta_k})\,\frac{d\theta_k}{2\pi}.
\end{equation}
The measure~\eqref{eq:circ-ensemble} is also determinantal and induced by the projection onto the subspace of trigonometric polynomials of degree up to $n-1$ with weight $w$.

\begin{remark-notat*}
  	Slightly stretching the notation, we use the same symbol $Z$ for the normalizing constant in different orthogonal polynomial ensembles.
\end{remark-notat*}

\subsection{Palm measures of determinantal point processes}

For concreteness here we consider the case of a point process on an open subset $U\subset\mathbb{R}$.
Recall that the first correlation measure $\rho_1$ of a point process $\mathbb{P}$ is defined by the formula
\begin{equation*}
	\rho_1(B)=\mathbb{E}_{\mathbb{P}}\#_B.
\end{equation*}
Let $\mathbb{P}$ be a probability measure on $\Conf(U)$ admitting the first correlation measure. For a point $p\in U$, following Khintchine \cite{Khintchine}, we define the reduced Palm measure as the limit
\begin{equation*}
	\lim_{\varepsilon\to 0}\frac{\mathbb{P}(\,\cdot\,|\#_{(p-\varepsilon,p+\varepsilon)}\ge 1)}{\rho_1((p-\varepsilon,p+\varepsilon))},
\end{equation*}
with the particle in point $p$ being deleted. The reduced Palm measure of the point process~$\mathbb{P}$ at $p$ is roughly interpreted as a conditional measure of our process conditioned to contain a particle at the point $p$. A general formalism of the Palm measures, based on the Campbell measures, has been developed by Kallenberg \cite{Kallenberg}; for a brief presentation see \cite{Buf-AnnProb}.

The Shirai--Takahashi theorem states that the Palm measure of a determinantal point process is determinantal and we have $(\mathbb{P}_K)^p=\mathbb{P}_{K^p}$, where
\begin{equation}\label{eq:Palm-meas}
	K^p(x,y)=K(x,y)-\frac{K(x,p)K(p,y)}{K(p,p)}.
\end{equation}
If $K$ is the projection onto a subspace $H$, then $K^p$ is the projection onto the subspace $\{f\in H: f(p)=0\}$.

If the additional condition 
\begin{equation}\label{eq:fin-inf}
	\rho_1(\{x:|x|\ge 1\})<+\infty
\end{equation}
is satisfied, then the Palm measure at $\infty$ is defined as the limit of conditional measures
\begin{equation*}
	\lim_{R\to\infty}
	\frac{\mathbb{P}(\,\cdot\,|\#_{\{x:|x|\ge R\}}\ge 1)}{\rho_1(\{x:|x|\ge R\})}.
\end{equation*}
In other words, condition \eqref{eq:fin-inf} guarantees that our point process is well-defined on the phase space $U\sqcup\{\infty\}$. The Palm measure is then  taken in this enlarged phase space.

\subsection{Admissible weights}

A function $\rho$ defined on a subset of~$\mathbb{R}$ is called \emph{an admissible weight} if the following conditions are satisfied. First, the function $\rho$ is continuous and positive on an open subset $U\subset\mathbb{R}$ except, perhaps, at finitely many points, where the function $\rho$ is allowed to have jump discontinuities or assume values~$0$ or $\infty$. In the neighbourhood of a point $p$ such that $\rho(p)=\infty$ we make the requirement
\begin{equation*}
	\int_{p-\varepsilon}^{p+\varepsilon}p(t)\,dt<+\infty.
\end{equation*}
We assume that the weight $\rho$ has limits $\lim\limits_{t\to+\infty}\rho(t)$, $\lim\limits_{t\to-\infty}\rho(t)$, possibly infinite, and that
\begin{equation*}
	\int_{\substack{|t|>1,\\t\in U}}\frac{\rho(t)}{t^2}dt<+\infty.
\end{equation*}
We now consider kernels $\Pi(x, y)$ of the form
\begin{equation*}
	\Pi(x,y)=\rho(x)\rho(y)\tilde{\Pi}(x,y),
\end{equation*}
where the kernel $\tilde{\Pi}(x,y)$ is a continuous function. If $\tilde\Pi(p,p)> 0$, then the Palm kernel is written in the form
\begin{equation*}
	\Pi^p(x,y)=\rho(x)\rho(y)\biggl(\tilde{\Pi}(x,y)-\frac{\tilde{\Pi}(x,p)\tilde{\Pi}(p,y)}{\tilde{\Pi}(p,p)}\biggr).
\end{equation*}
Our assumptions imply that the Palm projections $\Pi^q$ are continuous in the locally trace class topology. In other words, for any interval $[a, b]\subset U$ the family of operators $\chi_{[a,b]}\Pi^q\chi_{[a,b]}$ is continuous in $q$ in the space of trace-class operators. In particular, the required continuity still holds in points $q$ with $\rho(q)=\infty$. The continuity of the operators $\Pi^q$ implies the continuity of the Palm measures $\mathbb{P}_{\Pi^q}$ with respect to $q\in U$.

\subsection{Change of variables in kernel of determinantal processes}

A homeomorphism $g$ of the space $E$ induces a homeomorphism on the space of configurations $\operatorname{Conf}(E)$ by the formula $g(X)=\{g(x):x\in X\}$. Slightly stretching notation, we denote the induced homeomorphism by the same symbol~$g$.
Let the symbol $\mathbb{P}\circ g$ stand for the measure defined by the formula $\mathbb{P}\circ g(A)=\mathbb{P}(g(A))$. Consider the determinantal process $\mathbb{P}_K$ defined by an integral operator $K$ acting on the space $L_2(E, \mu)$ and admitting a kernel $K(x, y)$. Assume that a homeomorphism $g$ preserves the measure class of our reference measure $\mu$. One directly checks that the measure $\mathbb{P}_K\circ g$ is again determinantal with the kernel
\begin{equation*}
	g_*K(x,y)=K(g(x),g(y))\cdot 
	\sqrt{\dfrac{d\mu\circ g}{d\mu}(x)\dfrac{d\mu\circ g}{d\mu}(y)}.
\end{equation*}

Consider the $n$-particle orthogonal polynomial circular ensemble
\begin{equation}\label{eq:circ-orth}
	Z^{-1}\cdot \prod_{1\le k<l\le n}\bigl|e^{i\theta_k}-e^{i\theta_l}\bigr|^2\cdot
	\prod_{k=1}^n \frac{w(e^{i\theta_k})}{2\pi}d\theta_k,
\end{equation}
with the weight $w$ satisfying
\begin{equation*}
	\int |w(e^{i\theta})|\,\frac{d\theta}{2\pi}<+\infty.
\end{equation*}
The Palm measure at $1$ of our orthogonal polynomial ensemble is naturally identified with the $(n-1)$-particle orthogonal polynomial ensemble
\begin{equation}\label{eq:circ-orth-palm}
	Z^{-1}\cdot \prod_{1\le k<l\le n-1} \bigl|e^{i\theta_k}-e^{i\theta_l}\bigr|^2\cdot
	\prod_{k=1}^{n-1}\bigl|1-e^{i\theta_k}\bigr|^2 \frac{w(e^{i\theta_k})}{2\pi}d\theta_k.		
\end{equation}
Indeed, the Palm measure corresponds to the projection onto the subspace
\begin{equation*}
	\operatorname{span}\{1,e^{i\theta},\dots,e^{i(n-2)\theta}\}(1-e^{i\theta})\sqrt{w(e^{i\theta})},
\end{equation*}
and the $n-1$-particle orthogonal polynomial ensemble is represented by the projection onto subspace
\begin{equation*}
	\operatorname{span}\{1,e^{i\theta},\dots,e^{i(n-2)\theta}\}|1-e^{i\theta}|\sqrt{w(e^{i\theta})}.
\end{equation*}
These subspaces differ by a multiplication by a function with the absolute value equal to $1$. Thus the respective determinantal measures coincide.

The change of variable
\begin{equation}\label{eq:ch-var-theta-x}
	e^{i\theta_k}=\frac{x_k-i}{x_k+i},
\end{equation}
yields the following proposition.

\begin{proposition}\label{prop:tildeW}
  	Let $\tilde w$ be a positive Borel function on~$\mathbb{R}$ such that
	\begin{equation*}
		\int_{\mathbb{R}}\frac{\tilde w(x)\,dx}{1+x^2}<+\infty.
	\end{equation*}
	Consider the $n$-particle orthogonal polynomial ensemble
	\begin{equation}\label{eq:tildew-ens}
		Z^{-1}\cdot \prod_{1\le j<k\le n}(x_j-x_k)^2 \cdot
		\prod_{k=1}^n \frac{\tilde w(x_k)\,dx_k}{(1+x_k^2)^n}.
	\end{equation}
	The Palm measure at $\infty$ of the orthogonal polynomial ensemble~\eqref{eq:tildew-ens} is the $(n-1)$-particle orthogonal polynomial ensemble
	\begin{equation*}
		Z^{-1}\cdot \prod_{1\le j<k\le n-1}(x_j-x_k)^2 \cdot
		\prod_{k=1}^{n-1} \frac{\tilde w(x_k)\,dx_k}{(1+x_k^2)^n}.
	\end{equation*}
\end{proposition}

\begin{proof}
  Indeed, the change of variable~\eqref{eq:ch-var-theta-x} takes the first and the second orthogonal ensembles into the $n$-particle orthogonal polynomial ensemble~\eqref{eq:circ-orth} and the $(n-1)$-particle orthogonal polynomial ensemble~\eqref{eq:circ-orth-palm} respectively, and the proposition is now clear.
\end{proof}

\subsection{Hua---Pickrell measures and determinantal point processes}

We go back to the Hua---Pickrell measures $\mu^{(s)}$ and their decomposing measures $\mathbb{P}^{(s)}$. Borodin and Olshanski proved that $\mathbb{P}^{(s)}$ is a determinantal point process on $\mathbb{R}\setminus\{0\}$ with the kernel
\begin{equation}\label{eq:conf-kernel}
	K^{(s)}(x,y)=c_s\cdot\frac{P_s(x)Q_s(y)-P_s(y)Q_s(x)}{x-y},
\end{equation}
where
\begin{equation}\label{eq:PQc}
\begin{gathered}
	P_s(x)=\biggl|\frac{2}{x}\biggr|^{\re s}\cdot e^{-\tfrac{i}{x}+\pi\tfrac{\im s\cdot\sgn(x)}{2}}\cdot
	{}_1F_1\biggl[\genfrac{}{}{0pt}{}{s}{2\re s+1}\biggm|\frac{2i}{x}\biggr],\\
	Q_s(x)=\frac{2}{x}\cdot\biggl|\frac{2}{x}\biggr|^{\re s}\cdot e^{-\tfrac{i}{x}+\pi\tfrac{\im s\cdot\sgn(x)}{2}}\cdot
	{}_1F_1\biggl[\genfrac{}{}{0pt}{}{s}{2\re s+2}\biggm|\frac{2i}{x}\biggr],\\
	c_s=\frac{1}{2\pi}\Gamma\biggl[\genfrac{}{}{0pt}{}{s+1,\bar{s}+1}{2\re s+1, \re s+2}\biggr].
\end{gathered}
\end{equation}
Set
\begin{gather*}
  	\tilde{\rho}(x)=\biggl|\frac{2}{x}\biggr|^{\re s}\cdot 
	e^{\pi\tfrac{\im s\cdot\sgn(x)}{2}},\\
	\tilde A(x)=e^{-\tfrac{i}{x}}\cdot
	{}_1F_1\biggl[\genfrac{}{}{0pt}{}{s}{2\re s+1}\biggm|\frac{2i}{x}\biggr],\quad
	\tilde B(x)=\frac{2}{x}\cdot e^{-\tfrac{i}{x}}\cdot
	{}_1F_1\biggl[\genfrac{}{}{0pt}{}{s}{2\re s+2}\biggm|\frac{2i}{x}\biggr].
\end{gather*}
The confluent hypergeometric kernel is then expressed by the formula
\begin{equation*}
	K^{(s)}(x,y)=\tilde{\rho}(x)\tilde{\rho}(y)\frac{\tilde{A}(x)\tilde{B}(y)-\tilde{A}(y)\tilde{B}(x)}{x-y}.
\end{equation*}
Denote
\begin{equation*}
	\tilde{K}(x,y)=\frac{\tilde{A}(x)\tilde{B}(y)-\tilde{A}(y)\tilde{B}(x)}{x-y}.
\end{equation*}
On the diagonal $x=y$ the values of the kernel $\tilde{K}$ are given by the formula
\begin{equation*}
	\tilde{K}(x,x)=\tilde{A}'(x)\tilde{B}(x)-\tilde{A}(x)\tilde{B}'(x).
\end{equation*}
The functions $\tilde A$, $\tilde B$, $\tilde K(x,x)$ are holomorphic on $\mathbb{C}\setminus\{0\}\cup\{\infty\}$.
Apply the change of the variable $y = 1/x$ and set
\begin{equation*}
	\Pi^{(s)}(y,y)=\frac{1}{x^2}\tilde K^{(s)}\biggl(\frac{1}{x},\frac{1}{x}\biggr).
\end{equation*}

\begin{proposition}
	For any $s\in\mathbb{C}$, $\re s>-1/2$, we have $\Pi^{(s)}(0,0)>0$.
\end{proposition}

\begin{proof}
  	The change of variable $y=1/x$ preserves the integrable form of the kernel. Indeed, let
	\begin{gather*}
		\mathring{A}(y)=y\tilde{B}(x),\quad
		\mathring{B}(y)=y\tilde{A}(x),\quad
		\mathring{\rho}(y)=\frac{1}{|y|}.
	\end{gather*}
	In restriction to the diagonal $x=y$, $x,y\in\mathbb{R}$, the kernel $\Pi^{(s)}$ coincides with the restriction onto the diagonal of the kernel $K^{\lozenge}$ defined by the formula
	\begin{gather*}
		K^{\lozenge}(x,y)=\frac{A^{\lozenge}(x)B^{\lozenge}(y)-A^{\lozenge}(y)B^{\lozenge}(x)}{x-y},\\
		A^{\lozenge}(y)=\tilde{B}(x),\quad
		B^{\lozenge}(y)=\tilde{A}(x).
	\end{gather*}
	Hence
	\begin{align*}
		A^{\lozenge}(y)&{}=2ye^{-iy}\cdot
		{}_1F_1\biggl[\genfrac{}{}{0pt}{}{s}{2\re s+2}\biggm|2iy\biggr],\\
		B^{\lozenge}(y)&{}=e^{-iy}\cdot
		{}_1F_1\biggl[\genfrac{}{}{0pt}{}{s}{2\re s+1}\biggm|2iy\biggr],\quad
	\end{align*}
	One can see that $B^{\lozenge}(0)=1$, $A^{\lozenge}(0)=0$,
	$(A^{\lozenge})'(0)=2$, hence $\Pi^{(s)}(0,0)=K^{\lozenge}(0,0)=2$.
\end{proof}

The kernel $K^{(s)}$ is called the confluent hypergeometric kernel. We thus have the identity $\mathbb{P}^{(s)}=\mathbb{P}_{K^{(s)}}$.
The main result of the paper is now formulated as follows.

\begin{theorem}\label{thm:main}
  For any $s\in\mathbb{C}$, $\re s>-1/2$, the measure $\mathbb{P}_{K^{(s+1)}}$ is the Palm measure at~$\infty$ of the measure $\mathbb{P}_{K^{(s)}}$.
\end{theorem}

\begin{remark*}
	Recall that after the change of variable $x\mapsto 1/x$ the transformed kernel takes the form 
	\begin{equation}\label{eq:Pi-s}
		\Pi^{(s)}(x,y)=K^{(s)}\biggl(\frac{1}{x},\frac{1}{y}\biggr)\cdot \frac{1}{xy}.
	\end{equation}
	Consider the corresponding determinantal process $\mathbb{P}_{\Pi^{(s)}}$. Theorem~\ref{thm:main} is then equivalent to the statement that $\mathbb{P}_{\Pi^{(s+1)}}$ is the Palm measure of $\mathbb{P}_{\Pi^{(s)}}$ at $0$.
\end{remark*}

Borodin and Olshanski obtain the confluent hypergeometric kernel as the limit as $n\to\infty$ of the Christoffel---Darboux kernels of the orthogonal polynomial ensembles with the weight
\begin{equation}\label{eq:wns}
	w_{n,s}(x)=(1+x^2)^{-n-\operatorname{Re}s}\cdot e^{2\operatorname{Im}s\operatorname{Arg}(1+ix)}.
\end{equation}
Let $K_n^{(s)}$ be the $n$-th Chrisoffel---Darboux kernel of the orthogonal polynomials with weight $w_{n,s}$ defined by~\eqref{eq:wns}, which is the kernel of the projection of $L_2(\mathbb{R})$ onto
\begin{equation*}
	\operatorname{span}\{1,x,\dots,x^{n-1}\}\sqrt{w_{n,s}(x)}.
\end{equation*}
Borodin and Olshanski prove that
\begin{equation*}
	\lim_{n\to\infty}\sgn(x)\sgn(y)\cdot n\cdot K_n^{(s)}(nx,ny)=K^{(s)}(x,y)
\end{equation*}
(see the first line in the proof of Theorem 2.1 in \cite{BO}).
Moreover, Borodin and Olshanski in fact prove the convergence of the functions $A, B$ in the integrable representation of our kernels, writing
\begin{gather*}
	\sgn(x)\sgn(y)\cdot n\cdot K_n^{(s)}(nx,ny)=
	\tilde{\rho}_n(x)\tilde{\rho}_n(y)
	\frac{\tilde A_n(x)\tilde B_n(y)-\tilde A_n(y)\tilde B_n(x)}{x-y},\\
	K^{(s)}(x,y)=\tilde{\rho}(x)\tilde{\rho}(y)
	\frac{\tilde A(x)\tilde B(y)-\tilde A(y)\tilde B(x)}{x-y}.
\end{gather*}
Borodin and Olshanski establish the convergence $\tilde A_n\to\tilde A$, $\tilde B_n\to\tilde B$,
$\tilde \rho_n\to\tilde \rho$ as $n\to\infty$, uniform on compact sets. 
As we shall now see, this stronger convergence implies that the confluent hypergeometric kernel induces an orthogonal projection.

\begin{proposition}
  	The kernel $K^{(s)}$ induces a projection acting in $L_2(\mathbb{R})$.
\end{proposition}

The formula for the kernel $K^{(s)}$ shows that any function $f\in\operatorname{Ran} K^{(s)}$ can be written as 
\begin{equation*}
	f(x)=h(x)|x|^{-\re s}e^{\pi\cdot\tfrac{\im s\cdot\sgn(x)}{2}}, 	
\end{equation*}
where $h(x)$ is holomorphic on a neighbourhood of $\mathbb{R}\cup\{\infty\}\setminus\{0\}$. In particular, $h$ is holomorphic in a neighbourhood of $\infty$.
	
Let $H^{(s)}=\operatorname{Ran}K^{(s)}$. Consider the following subspaces  in $H^{(s)}$. Set $H^{(s,0)}=H^{(s)}$ and 
\begin{multline*}
	H^{(s,n)}=\biggl\{f\in H^{(s)}: f(x)=\frac{h(x)}{|x|^{\re s}}e^{\pi\cdot\tfrac{\im s\cdot\sgn(x)}{2}},\\ h(z)=O(1/|z|^n)\text{ as }z\to\infty\biggr\},\quad n=1,2,\dots
\end{multline*}
Denote $H^{(s,n)}\ominus H^{(s,n+1)}$ by $L^{(s,n)}$. We will show below that $\dim L^{(s,n)}=1$ and there exists a function $\varphi(x)$, $|\varphi(x)|\equiv 1$, that satisfies $H^{(s,1)}=\varphi\cdot H^{(s+1)}$.

We arrive at the hierarchical decomposition
\begin{equation*}
	H^{(s)}=L^{(s,0)}\oplus L^{(s,1)}\oplus\dots\oplus L^{(s,k)}\oplus\cdots,
\end{equation*}
where for any $l\in\mathbb{N}$ subspace $H^{(s+l)}$ is obtained from the subspace $L^{(s,l)}\oplus L^{(s,l+1)}\oplus\cdots$ by multiplication by a function of absolute value~$1$.

\section{The proof of Theorem~\ref{thm:main}}

Let $U\subset\mathbb{R}$ be an open set endowed with a $\sigma$-finite Borel measure $\mu$. Let $H\subset L_2(U, \mu)$ be a closed subspace such that the corresponding operator of orthogonal projection onto it admits an integrable kernel, that is, a kernel of the form
\begin{equation}\label{eq:int-kernel}
	\Pi(x,y)=\frac{A(x)B(y)-A(y)B(x)}{x-y}\rho(x)\rho(y),
\end{equation}
where $A, B$ are smooth functions and $\rho$ is continuous and positive. By the Christoffel---Darboux formula, the projections given by the orthogonal polynomial ensembles have integrable kernels.

If $\Pi(p, p)>0$, then the Palm kernel $\Pi^p$ is also integrable. Define
\begin{equation}\label{eq:ApBp-nonzero}
	A^p(x)=A(x)+A(p)\frac{\Pi(p,x)}{\Pi(p,p)},\quad 
	B^p(x)=B(x)+B(p)\frac{\Pi(p,x)}{\Pi(p,p)}.
\end{equation}

\begin{proposition}
  	Let $p \in U$ satisfy $\Pi(p, p)>0$.
	Then the kernel $\Pi^p$ is integrable. In particular, if $A^p$, $B^p$ are defined as in \eqref{eq:ApBp-nonzero}, then
	\begin{equation}\label{eq:Pi-p-nonzero-case}
		\Pi^p(x,y)=\Pi(x,y)-\frac{\Pi(x,p)\Pi(y,p)}{\Pi(p,p)}=
		\frac{A^p(x)B^p(y)-A^p(y)B^p(x)}{x-y}.
	\end{equation}
\end{proposition}

\begin{remark*}
  If $A(p)=0$, $B(p)\ne 0$, then the kernel $\Pi^p$ is expressed as follows (see~\cite{Buf-AnnProb})
	\begin{equation*}
		\Pi^p(x,y)=\rho(x)\rho(y)\biggl(\frac{A(x)B(y)-A(y)B(x)}{x-y}-c(p)\frac{A(x)}{x-p}\frac{A(y)}{y-p}\biggr),
	\end{equation*}
	where $c(p)=B(p)^2\rho(p)^2/\Pi(p,p)=B(p)/A'(p)$, hence the kernel may be expressed as
	\begin{equation}\label{eq:Palm-int}
		\frac{\Pi^p(x,y)}{\rho(x)\rho(y)}=\frac{A(x)\Bigl(B(y)-c(p)\frac{A(y)}{y-p}\Bigr)-
			A(y)\Bigl(B(x)-c(p)\frac{A(x)}{x-p}\Bigr)}{x-y}.
	\end{equation}
\end{remark*}

We now consider the case $\Pi(p, p)=0$. If there exists a function $u(x)$ satisfying
\begin{equation*}
	\Pi(x,y)=\rho(x)\rho(y)u(x)u(y)\cdot\tilde{\Pi}(x,y),\quad
	\tilde{\Pi}(x,y)=\frac{\tilde A(x)\tilde B(y)-\tilde A(y)\tilde B(x)}{x-y},
\end{equation*}
where
\begin{equation*}
	\tilde A(x)=\frac{A(x)}{u(x)},\quad 
	\tilde B(x)=\frac{B(x)}{u(x)}
\end{equation*}
are smooth functions and
\begin{equation}\label{eq:cond-PiTilde-ne0}
  \tilde{\Pi}(p,p)=\tilde A(p)\tilde B'(p)-\tilde A'(p)\tilde B(p)> 0,
\end{equation}
then one can express the kernel as
\begin{equation}\label{eq:Pi-p-zero-case}
	\Pi^p(x,y)=\rho(x)\rho(y)u(x)u(y)\frac{A^p(x)B^p(y)-A^p(y)B^p(x)}{x-y},
\end{equation}
where
	\begin{equation}\label{eq:ApBp-zero}
	A^p(x)=\tilde A(x)+\tilde A(p)\frac{\tilde\Pi(p,x)}{\tilde\Pi(p,p)},\quad 
	B^p(x)=\tilde B(x)+\tilde B(p)\frac{\tilde\Pi(p,x)}{\tilde\Pi(p,p)}.
\end{equation}
Below we will show that the kernel $\Pi^p$ defined as in~\eqref{eq:Pi-p-zero-case} is the Palm kernel of the process $\mathbb{P}_\Pi$ at point~$p$.

Our next aim is to prove the existence of the function $u$ satisfying~\eqref{eq:cond-PiTilde-ne0} under the additional assumption that $A, B$ are holomorphic. We choose the function $u$ in the form $u(x)=(x-p)^k$ with an appropriate $k$. In addition, we will exclude the case when the orders of zeros in $p$ of the functions $A, B$ differ by more than one (see Proposition~\ref{prop:norm-form} below).

We thus consider determinantal point processes with integrable kernels of the form~\eqref{eq:int-kernel}, under assumptions of $A, B$ being holomorphic in neighborhoods $U$ and taking real values on the real axis. We also assume $\rho$ to be positive and continuous on $U$.

Diagonal values of $\Pi$ are defined by L'H\^opital's rule:
\begin{equation*}
	\Pi(x,x)=\bigl(A'(x)B(x)-A(x)B'(x)\bigr)\rho^2(x).
\end{equation*}
The functions $A, B, \rho$ are defined up to a transform $(A,B,\rho)\mapsto (Au,Bu,\rho/u)$, and the identity
\begin{equation*}
	(Au)'(x)(Bu)(x)-(Bu)'(x)(Au)(x)=(A'(x)B(x)-A(x)B'(x))u^2(x)
\end{equation*}
shows that the diagonal values are well-defined.

Our assumptions yield that any function $f\in H$ is expressed as $f=f_1\rho$, where $f_1$ is holomorphic in a neighbourhood $U$.

Integrable kernels have the following division property: if $\Pi(p,p)>0$ and $f\in H$ satisfies $f(p)=0$, then the function $f(t)/(t-p)$ belongs to~$H$. The division property only requires the assumption that the functions $A, B$ are smooth, but in our case they are actually holomorphic.

For integrable projection kernels of the form 
\begin{equation*}
	\Pi(x,y)=\rho(x)\rho(y)\frac{A(x)B(y)-A(y)B(x)}{x-y},
\end{equation*}
where $A, B$ are smooth, the division property takes the following form. Assume, as above, $\tilde{\Pi}(x,y)=\Pi(x,y)/\rho(x)\rho(y)$. If $p\in U$ satisfies $\tilde{\Pi}(p,p)> 0$ and $f\in H$ is expressed as $f=\rho h$, $h(p)=0$, then
$f(t)/(t-p)\in H$. Observe that if $A, B$ are holomorphic in a neighbourhood of $p$ and $A(p)=B(p)=0$, then the kernel $\Pi$ can be expressed as follows
\begin{multline*}
	\Pi(x,y)=\rho(x)\rho(y)(x-p)^k(y-p)^k\tilde{\Pi}(x,y)={}\\
	{}=\rho(x)\rho(y)(x-p)^k(y-p)^k \frac{\tilde A(x)\tilde B(y)-\tilde A(y)\tilde B(x)}{x-y},
\end{multline*}
where $\tilde{\Pi}(p,p)> 0$. Indeed, consider the largest $k$ such that any $f\in H$ can be decomposed as $f(t)=\rho(t)(t-p)^k h(t)$ for some holomorphic in a neighbourhood of $p$ function $h$. In this case the kernel
\begin{equation}\label{eq:Pi-order-of-0}
	\tilde{\Pi}(x,y)=\frac{\Pi(x,y)}{\rho(x)\rho(y)(x-p)^k(y-p)^k}
\end{equation}
satisfies $\tilde{\Pi}(p,p)>0$. The division property implies that if $f\in H$ has the form $f(t)=\rho(t)(t-p)^kh(t)$, $h(p)=0$, then $f(t)/(t-p)\in H$.

Let us now pass to the case $\Pi(p,p)=0$. We need several technical propositions.

\begin{proposition}
  Let $H$ be infinite-dimensional. Let $p\in\nobreak U$ and $k\ge 0$ be such that the kernel $\tilde{\Pi}$, defined by the formula \eqref{eq:Pi-order-of-0}, satisfies $\tilde\Pi(p,p)>0$.
  Then for all $n\ge k$ there exists a function $f\in H$, having a zero of order $n$ at the point $p$.
\end{proposition}

\begin{remark*}
  The number $k$ is the minimal order of zero of the function $f_1$, where $f=f_1\rho\in H$
\end{remark*}

\begin{proof}
  Define the map $\Phi\colon H\to\mathbb{C}^n$ as follows
	\begin{equation*}
		\Phi\colon f=f_1\rho\mapsto (f_1(p),f_1'(p),\dots,f_1^{(n-1)}(p)).
	\end{equation*}
	Since $H$ is infinite-dimensional, the subspace $\ker\Phi\simeq H/\operatorname{im}\Phi$ is also infinite-dimensional. In particular, it contains a non-zero function $g=g_1\rho$. Then the function $g_1$ has a zero of an order $m\ge n\ge k$ in the point $p$. The division property, applied $m-n$ times to the function~$g$, gives a function with a zero of order~$n$.
\end{proof}

\begin{proposition}\label{prop:norm-form}
  Let $p\in U$ satisfy $\Pi(p, p)=0$. Then there exists $k\in\mathbb{N}$, a neighbourhood $V$ of the point $p$ and functions $A, B, \rho$, providing an integrable representation of the kernel $\Pi$ in the form~\eqref{eq:int-kernel}, and such that 
	\begin{equation}\label{eq:A-breveA}
		A(x)=(x-p)^{k+1}\breve A(x),\quad B(x)=(x-p)^{k}\breve B(x),
	\end{equation}
	where the functions $\breve{A}(x)$, $\breve{B}(x)$ and $\Pi(x,y)/(x-p)^k(y-p)^k$ do not vanish for $x\in V$, $y\in V$. In particular, $\Pi(x,x)/(x-p)^{2k}>0$ for $x\in V$.
\end{proposition}

\begin{proof}
  If $\Pi(p, p)=0$, then all functions from $H$ vanish in $p$. Let $k$ be the minimal possible order of zero of a function from $H$ at $p$ and let $f\in H$ be defined by the formula $f(x)=(x-p)^k g(x)$, $g\in H$, $g(p)\ne 0$. Let also $\tilde\Pi$ be the kernel of the projection onto the orthogonal complement in $H$ of the function $f$. Then 
	\begin{equation*}
		\Pi(x,y)=(x-p)^k(y-p)^k g(x)\overline{g(y)}+\tilde\Pi(x,y) 
	\end{equation*}
	and
	\begin{equation*}
		\frac{\Pi(x,y)}{(x-p)^k(y-p)^k}\biggr|_{x=y=p}\ge |g(p)|^2>0.
	\end{equation*}
	Recall the explicit formulae for the functions $A, B$, providing an integrable structure of the kernel $\Pi$. We now use an argument from \cite{BufRom}. Take a point $q$ close to $p$ such that $\Pi(q, q)>0$. For an arbitrary $x\in U$ define a function $\varphi_x(t)$ by the formula
	\begin{equation*}
		\varphi_x(t)=\frac{\Pi(x,t)\Pi(q,q)-\Pi(x,q)\Pi(q,t)}{t-q}.
	\end{equation*}
	Now the functions $A$ and $B$ may be defined by the formula
	\begin{equation*}
		A(x)\rho(x)=(x-q)\Pi(x,q),\quad B(x)\rho(x)=\Pi(x,q)-(x-q)\frac{\varphi_x(q)}{\Pi(q,q)}.
	\end{equation*}
	The above formulae show that the function $\varphi_x(t)/\rho(x)$, considered in the $x$ variable for fixed $t$, is holomorphic and divisible by $(x-p)^k$. The same is true, of course, for the functions $\Pi(x, q)$ and, consequently, for the functions $A$ and $B$, so $(x-p)^k$ divides both $A$ and $B$. Our kernel is preserved by the linear transformation of the form $(A^\circ(x),B^\circ(x))=(A(x),B(x))C$, where $C$ is a constant $2\times 2$-matrix with $\det C=1$. An appropriate choice of $C$ yields a function $A$ divisible by $(x-\nobreak p)^{k+1}$. The condition 
	\begin{equation*}
		\Pi(x,y)/(x-p)^k(y-p)^k\ne 0, 	
	\end{equation*}
	satisfied for $x$ and~$y$ in sufficiently small neighbourhood of $p$, implies that $\breve A$ and $\breve B$ from \eqref{eq:A-breveA} both satisfy $\breve A(p)\ne 0$, $\breve B(p)\ne 0$. The proposition is proved.
\end{proof}

The next step is to establish that the Palm kernel $\Pi^p$ is well-defined even in the case $\Pi(p, p)=0$. If $u(x)=(x-p)^k$, then the functions $\tilde A(x)=A(x)/u(x)$, $\tilde B(x)=B(x)/u(x)$ are both continuous. Consequently, the functions $A^q$, $B^q$ and $\Pi^q$ may be defined by the formulae~\eqref{eq:ApBp-zero} and \eqref{eq:Pi-p-zero-case} for all $q$ close to $p$ and these functions depend continuously on $q$. Therefore, the measure $\mathbb{P}_{\Pi^p}$ is the limit of the measures $\mathbb{P}_{\Pi^q}$ for $q\to p$. On the other hand, if $u(q)\ne 0$, the expressions~\eqref{eq:Pi-p-zero-case} and \eqref{eq:Pi-p-nonzero-case} induce the same kernel $\Pi^q$, so the point process $\mathbb{P}_{\Pi^q}$ is the Palm measure of the point process $\mathbb{P}_\Pi$ at the point~$q$. The Palm measure depends continuously on the point, so the limit measure is the Palm measure at the point $p$. Thus the Palm measure $\mathbb{P}_{\Pi^p}$ is the determinantal point process with the integrable kernel~\eqref{eq:Pi-p-zero-case}.

Further, let us assume that the kernels $\Pi_n$ converge locally uniformly to $\Pi$ as $n\to\infty$. We need a somewhat stronger form of this convergence. Namely, let 
\begin{align*}
\Pi_n(x,y)&{}=\frac{A_n(x)B_n(y)-A_n(y)B_n(x)}{x-y}\rho_n(x)\rho_n(y),\\
\Pi(x,y)&{}=\frac{A(x)B(y)-A(y)B(x)}{x-y}\rho(x)\rho(y),
\end{align*}
where $A_n, B_n, A, B$ are holomorphic in a fixed neighbourhood $V$ of $U$, $\rho_n,\rho$ are admissible weights. Let $p_1,\dots, p_l$ be all the points, where the weight $\rho$ either goes to infinity of has jump discontinuity. Let us say that the sequence of kernels $\Pi_n$ $C$-converges to the kernel $\Pi$ as $n\to\infty$ if
\begin{enumerate}
	\item $(A_n)$ and $(B_n)$ converge locally uniformly on compact subsets of $V$,
	\item $(\rho_n)$ converge uniformly on compact subsets of $U\setminus\{p_1,\dots,p_l\}$, and there exists an open subset $W$, containing $p_1,\dots,p_l$ and satisfying
	\begin{equation*}
		\int_W \sup_n \rho_n\,d\mu<+\infty
	\end{equation*}
\end{enumerate}

\begin{remark*}
  In concrete examples, the functions $A_n$ and $B_n$ are usually taken to be of the form
	\begin{equation*}
		A_n(x)=c_n^{(1)}\mathcal{P}_n(x),\quad
		B_n(x)=c_n^{(2)}(\mathcal{P}_n(x)-c_n^{(3)}\mathcal{P}_{n-1}(x)),
	\end{equation*}
	where $\mathcal{P}_n$ are orthogonal polynomials and $c_n^{(1)}$, $c_n^{(2)}$, $c_n^{(3)}$ are constants.
\end{remark*}

The following proposition establishes sufficient conditions for a kernel $\Pi$ to be a projection kernel.

\begin{proposition}
  Let integrable projection kernels
	\begin{equation*}
		\Pi_n(x,y)=\rho_n(x)\rho_n(y)\tilde{\Pi}_n(x,y)=\rho_n(x)\rho_n(y)\frac{A_n(x)B_n(y)-A_n(y)B_n(x)}{x-y}
	\end{equation*}
	$C$-converge to the limit integrable kernel
	\begin{equation}\label{eq:Pi-tildePi}
		\Pi(x,y)=\rho(x)\rho(y)\tilde{\Pi}(x,y)=\rho(x)\rho(y)\frac{A(x)B(y)-A(y)B(x)}{x-y}.
	\end{equation}
	If for any $p\in U$
	\begin{equation}\label{eq:cond-sup}
		\int_U \sup_n |\Pi_n(p,y)|^2\,dy<+\infty,
	\end{equation}
	then $\Pi$ is a projection kernel.
\end{proposition}

\begin{proof}
  	Indeed, the equalities
	\begin{equation*}
		\Pi(p,p)=\int_U |\Pi(x,p)|^2\,dx\quad \text{and}\quad \Pi(x,y)=\int_U \Pi(x,u)\Pi(u,y)\,du
	\end{equation*}
	followed by the dominated convergence theorem.
\end{proof}

The condition~\eqref{eq:cond-sup} follows from the more explicit condition
	\begin{equation}\label{eq:cond-sup-AB}
	\int_U \sup_n \frac{A_n^2(x)+B_n^2(x)}{x^2+1}\rho_n^2(x)\,dx<+\infty
\end{equation}
We now check the condition~\eqref{eq:cond-sup-AB} for the confluent hypergeometric kernel.

\begin{corollary*}
  The confluent hypergeometric kernel is a projection kernel.
\end{corollary*}

\begin{proof}
  Borodin and Olshanski \cite{BO} obtained the confluent hypergeometric kernel by a limit transition from the Christoffel---Darboux kernels for the pseudo-Jacobi polynomials.
	
  The pseudo-Jacobi polynomials are expressed via hypergeometric functions (see~formulae (1.18) and (1.19) in \cite{BO}), and the convergence of the functions $A_n$, $B_n$ follows from the convergence
	\begin{equation}\label{eq:hypergeom-conv}
		\lim_{|n|\to\infty} {}_2F_1\biggl[\genfrac{}{}{0pt}{}{n\;b}{c}\biggm|\frac{z_n}{n}\biggr]=
		{}_1F_1\biggl[\genfrac{}{}{0pt}{}{b}{c}\biggm|z\biggr],
	\end{equation}
	where $b,c$ depend only on $s$, and $|nz_n-z|=O(z)$ uniformly in $z$ on any circle with center at the origin. The convergence~\eqref{eq:hypergeom-conv} is uniform on compact sets, as one can easily see, following Borodin and Olshanski, from the integral representations
	\begin{align*}
		{}_1F_1\biggl[\genfrac{}{}{0pt}{}{b}{c}\biggm|z\biggr]&{}=
		\frac{1}{\Gamma(b)\Gamma(c-b)}\int_0^1 e^{zt}t^{b-1}(1-t)^{c-b-1}\,dt,\\
		{}_2F_1\biggl[\genfrac{}{}{0pt}{}{a\;b}{c}\biggm|z\biggr]&{}=
		\frac{1}{\Gamma(b)\Gamma(c-b)}\int_0^1 \frac{t^{b-1}(1-t)^{c-b-1}}{(1-zt)^a}\,dt.		
	\end{align*}
	The integral~\eqref{eq:cond-sup-AB} is hence bounded from above by the convergent integral
	\begin{equation*}
		\int_{|x|>1}\frac{1}{x^2+1}\frac{dx}{|x|^{2\re s}}+
		\int_{-1}^{1}|x|^{2\re s}\,dx.
	\end{equation*}
	The projection property is proved.
\end{proof}

The following property of the $C$-convergence is immediate from the definitions.

\begin{proposition}\label{prop:C-conv1}
  Let $\Pi_n$ be a sequence of kernels $C$-converging to a limit kernel $\Pi$ as $n\to\infty$. For any $p\in\mathbb{R}$ satisfying $\tilde\Pi(p,p)>0$, where $\tilde{\Pi}$ is defined by the formula~\eqref{eq:Pi-tildePi}, the Palm kernels $\Pi_n^p$ represented in the integrable form by~\eqref{eq:Pi-p-nonzero-case}, $C$-converge to the limit Palm kernel $\Pi^p$.
\end{proposition}

\begin{proposition}\label{prop:C-conv-Palm}
  Let $\mathbb{P}^{(n)}=\mathbb{P}_{\Pi_n}$ be a sequence of orthogonal polynomial ensembles of degree $n$ with weights $w_n(x)$. Assume that the kernels $\Pi_n$ $C$-converge to the limit projection kernel $\Pi$ as $n\to\infty$, and $\tilde{\Pi}(p,p)>0$ holds, where $\tilde{\Pi}$ is defined by~\eqref{eq:Pi-tildePi}. Then for any $p\in\mathbb{R}$ the orthogonal polynomial ensembles of degree $n-1$ with the weight $(x-p)^2w_n(x)$ converge to the Palm measure $\mathbb{P}_{\Pi^p}$ at the point~$p$.
\end{proposition}

\begin{proof}
  Indeed, the Palm kernels $\Pi_n^p$ $C$-converge to the kernel $\Pi^p$. The Palm kernel corresponds to the subspace 
	\begin{equation*}
		\operatorname{span}\{1,x,\dots,x^{n-2}\}(x-p)\sqrt{w_n(x)}
	\end{equation*}
	of functions from $\operatorname{Ran}\Pi_n$ vanishing in the point $p$.
	The orthogonal polynomial ensemble of degree $n-1$ with the weight $(x-\nobreak p)^2w_n(x)$ corresponds to the subspace
	\begin{equation*}
		\operatorname{span}\{1,x,\dots,x^{n-2}\}|x-p|\sqrt{w_n(x)}. 
	\end{equation*}
	These two subspaces differ by multiplication by a function with absolute value of 1, so the respective determinantal measures coincide. The proof is complete.
\end{proof}

Now recall that the kernel $K_n^{(s)}$ is the $n$-th Christoffel---Darboux kernel of the orthogonal polynomial ensemble with the weight $w_{n,s}$, given by the formula~\eqref{eq:wns}. Proposition~\ref{prop:tildeW} implies that $\mathbb{P}_{K^{(s+1)}_{n-1}}$ is the Palm measure at $\infty$ of the measure $\mathbb{P}_{K^{(s)}_{n}}$. We apply the change of variable $x\mapsto 1/x$:
\begin{equation*}
	\Pi_n^{(s)}(x,y)=K_n^{(s)}\biggl(\frac{1}{x},\frac{1}{y}\biggr) \cdot \frac{1}{xy}.
\end{equation*}
It is clear that $\mathbb{P}_{\Pi^{(s+1)}_{n-1}}$ is the Palm measure at $0$ of the measure $\mathbb{P}_{\Pi^{(s)}_{n}}$. The sequence $(1/n)\Pi^{(s)}_n(x/n,y/n)$ $C$-converges as $n\to\infty$ to the kernel $\Pi^{(s)}(x,y)$. Therefore, $\mathbb{P}_{\Pi^{(s+1)}}$ is the Palm measure at $0$ of the measure $\mathbb{P}_{\Pi^{(s)}}$. The inverse change of variable yields that $\mathbb{P}_{K^{(s+1)}}$ is the Palm measure at $\infty$ of the measure $\mathbb{P}_{K^{(s)}}$. Theorem~\ref{thm:main} is proven completely.

Consider a closed subspace $H^{(p,n)}$ of $H$, containing all functions with an order of a zero in $p$ not less than $n$. For any $p\in U$ there exists $n_0$ such that for any $n>n_0$ we have
\begin{equation*}
	\dim (H^{(p,n)}\ominus H^{(p,n+1)})=1.
\end{equation*}
Denoting $L^{(p,n)}=H^{(p,n)}\ominus H^{(p,n+1)}$, we arrive at the hierarchical decomposition
\begin{equation*}
	H=L^{(p,0)}\oplus L^{(p,1)}\oplus\dots\oplus L^{(p,n)}\oplus\cdots.
\end{equation*}

\section{Concluding remarks: convergence of circular orthogonal polynomial ensembles}

Another proof of Theorem~\ref{thm:main} can be given using the scaling limit transition from the orthogonal polynomials on the circle to the orthogonal polynomials on the line under the linear scaling $\theta\to\theta/n$.

Assume that for any integer $n$ there is a weight~$w_n$ on the unit circle. Let $\Pi_n(z,\zeta)$ be a Christoffel---Darboux kernel of the orthogonal projection onto the subspace of the polynomials of degree at most $n$ with respect to the weight~$w_n$. The Christoffel---Darboux formula implies that the kernel $\Pi_n$ can be expressed by the formula
\begin{equation*}
\Pi_n(z,\zeta)=
\frac{\overline{\varphi^*_{n+1}(\zeta)}\cdot \varphi^*_{n+1}(z)-
\overline{\vphantom{\varphi^*}\varphi_{n+1}(\zeta)}\cdot \varphi_{n+1}(z)}{1-\overline{\zeta}z}\cdot\sqrt{w_n(z)w_n(\zeta)},
\end{equation*}
where $\varphi_n$ are the orthogonal polynomials and $\varphi^*_n$ is dual to $\varphi_n$: $\varphi_n^*(z)=z^n\overline{\varphi_n(1/\bar z)}$.

As before, the key r\^ole is played by the integrable form of the limit kernel~$\Pi$.

Let $\Pi_n$ be a sequence of kernels on the circle defined by the formula
\begin{equation*}
	\Pi_n(z,\zeta)=\tilde\Pi_n(z,\zeta)\rho_n(z)\rho_n(\zeta)=\frac{\overline{Q_n^*(\zeta)}Q_n^*(z)-\overline{Q_n(\zeta)}Q_n(z)}{1-\overline{\zeta}z}\cdot\rho_n(z)\rho_n(\zeta),
\end{equation*}
where $Q_n$ is a polynomial and $\rho_n$ is a positive function. Assume the following limit relations:
\begin{equation*}
	\lim_{n\to\infty} Q_n(e^{it/n})=Q(t),\quad
	\lim_{n\to\infty} n\rho_n(e^{it/n})=\rho(t).	
\end{equation*}
Of course, in this case we must also have
\begin{equation*}
	\lim_{n\to\infty} Q^*_n(e^{it/n})=Q^*(t).	
\end{equation*}
Under these conditions we say that the sequence of kernels $\Pi_n$ on the circle $T$-converges to the limit kernel $\Pi$ on the line. The following statement is clear from the definitions.
\begin{proposition}\label{prop:T-conv}
  If a sequence of kernels $\Pi_n$ $T$-converges to $\Pi$ as $n\to\infty$ and $\tilde\Pi(1,1)>0$, then the sequence of the Palm kernels $\Pi_n^{(1)}$ $T$-converges to $\Pi^{(1)}$ as $n\to\infty$.
\end{proposition}

A counterpart of Proposition~\ref{prop:C-conv-Palm} also holds for $T$-convergence of projection kernels.

\begin{proposition}
  Consider a sequence of orthogonal polynomial ensembles $\mathbb{P}_{\Pi_n}$ of degree $n$ with weights $w_n$ on the circle. Assume that the kernels $\Pi_n$ $T$-converge to the limit projection kernel $\Pi$ as $n\to\infty$. Assume furthermore that $\tilde\Pi(1,1)>0$. Then the orthogonal polynomial ensembles of degree $n-1$ with the weight $|1-\nobreak e^{it}|^2w_n(e^{it})$ converge to the measure $\mathbb{P}_{\Pi^0}$, that is, the Palm measure at zero of the determinantal measure $\mathbb{P}_\Pi$.
\end{proposition}

The proof is similar to the one for $C$-convergence and relies on the invariance of the $T$-convergence under taking Palm measures, as well as on the equality of the determinantal measures corresponding to the subspaces
\begin{gather*}
	\operatorname{span}\{1,e^{it},\dots,e^{i(n-2)t}\}(1-e^{it})\sqrt{w(t)}\\
	\intertext{and}
	\operatorname{span}\{1,e^{it},\dots,e^{i(n-2)t}\}|1-e^{it}|\sqrt{w(t)}.
\end{gather*}
Observe also that the $T$-convergence is related to the $C$-convergence by a linear fractional variable change.

The circle $\{z:|z|=1\}$ is mapped onto the real line by the change of variables $z=(x-i)/(x+i)$. The Christoffel---Darboux kernels $\Pi_n(z,\zeta)$ are then transformed to the kernels 
\begin{equation*}
	\tilde\Pi_n(x,y)=\frac{1}{\sqrt{x^2+1}\sqrt{y^2+1}}
	\Pi_n\biggl(\frac{x-i}{x+i},\frac{y-i}{y+i}\biggr).
\end{equation*}
(The factor of $1/\sqrt{x^2+1}\sqrt{y^2+1}$ corresponds to the Jacobian of the change of variables.)

Then one can find a function $\varphi_n(x)$, $|\varphi_n(x)|\equiv 1$, such that the $T$-convergence of the kernels $\Pi_n$ implies $C$-convergence of the kernels $\tilde\Pi_n(x,y)\varphi_n(x)\overline{\varphi_n(y)}$.

Theorem~\ref{thm:main} follows now from Proposition~\ref{prop:T-conv} and the representation due to Bourgade, Nikeghbali and Rouault~\cite{BNR} of the determinantal process with the confluent hypergeometric kernel $K^{(s)}$ (see the formula~\eqref{eq:conf-kernel}) as a limit of the Jacobi circular orthogonal polynomial ensembles.

Indeed, Bourgade, Nikeghbali and Rouault in~\cite{BNR} consider an $n$-th orthogonal polynomial ensemble on the circle, corresponding to the weight
\begin{equation*}
	w^{(s)}(\theta)=(1-e^{i\theta})^s(1-e^{-i\theta})^{\bar s},
\end{equation*}
and prove that under scaling $\theta=x/n$ the respective sequence of determinantal measures with finite number of particles converges as $n\to\infty$ to the determinantal process on the line with the confluent hypergeometric kernel $\Pi^{(s)}$ defined by the formula~\eqref{eq:Pi-s}. Proposition~\ref{prop:T-conv} now implies Theorem~\ref{thm:main}.\qed


\begin{thebibliography}{99}
	\bibitem{AS} M. Abramowitz and I. Stegun, Handbook of Mathematical Functions,  National Bureau of Standards,
	Department of Commerce of the United States of America, Tenth Printing, 1972.
	
	\bibitem{baruch} E.M. Baruch, The classical Hankel transform
	in the Kirillov model,  arxiv:1010.5184v1.
	
	\bibitem{Bogachev} V.I. Bogachev, Measure theory. Vol. II. Springer Verlag, Berlin, 2007.
	
	\bibitem {BorOx} A.M. Borodin, Determinantal point processes, in
	The Oxford Handbook of Random Matrix Theory, Oxford University Press, 2011.
	
	\bibitem{BO} A. Borodin, G. Olshanski, Infinite random matrices and ergodic measures.
	Comm. Math. Phys. 223 (2001), no. 1, 87--123.
	
	\bibitem{BorRains}A.M. Borodin, E.M. Rains, Eynard-Mehta theorem, Schur process, and their pfaffian analogs.
	J. Stat. Phys. 121 (2005), 291--317.
	
	\bibitem{BNR}P. Bourgade, A. Nikeghbali, A. Rouault,  Ewens Measures on Compact Groups and Hypergeometric Kernels, S{\'e}minaire de Probabilit{\'e}s XLIII, Springer
	Lecture Notes in Mathematics 2011, pp 351-377.
	
	\bibitem{Buf-ergdec} A. I. Bufetov, Ergodic decomposition for measures quasi-invariant under a Borel action of an inductively compact group, Sb. Math., 205:2 (2014), 192–219.
	
	\bibitem{Buf-inferg}Alexander I. Bufetov, Finiteness of Ergodic Unitarily Invariant Measures on Spaces of Infinite Matrices, Ann. Inst. Fourier (Grenoble), 64:3 (2014), 893--907.
	
	\bibitem{Buf-TrMIAN}A. I. Bufetov, A Palm Hierarchy for Determinantal Point Processes with the Bessel Kernel, Proc. Steklov Inst. Math., 297 (2017), 90–97
	
	\bibitem{Buf-AnnProb}Alexander I. Bufetov, Quasi-symmetries of determinantal point processes, Ann. Probab., 46:2 (2018), 956--1003.
	
	\bibitem{Buf-umn} A.I. Bufetov, Multiplicative functionals of determinantal processes,
	Uspekhi Mat. Nauk 67 (2012), no. 1 (403), 177--178;
	translation in Russian Math. Surveys 67 (2012), no. 1, 181--182.
	
	\bibitem{Buf-CIRM} A. Bufetov, Infinite determinantal measures, Electronic Research Announcements in the Mathematical Sciences, 20 (2013), pp. 8 -- 20.
	
	\bibitem{BufOlsh-Studia}Alexander I. Bufetov, Grigori Olshanski, A hierarchy of Palm measures for determinantal point processes with gamma kernels, Studia Math., 267 (2022), 121--160
	
	\bibitem{BufRom}A. I. Bufetov, R. V. Romanov. Division subspaces and integrable kernels, Bull. Lond. Math. Soc., 51:2 (2019), 267--277.
	
	\bibitem{DVJ} D.J.Daley, D. Vere-Jones, An introduction to the theory of point processes, vol.~I--II, Springer Verlag, 2008.
	
	\bibitem{Faraut} J. Faraut, Analyse sur les groupes de Lie: une introduction, Calvage et Mounet 2006.
	
	\bibitem{Faraut-angl} J. Faraut, Analysis on Lie Groups, Cambridge University Press, 2008.
	
	\bibitem{ghobber1}
	S. Ghobber,  Ph. Jaming, Strong annihilating pairs for the Fourier-Bessel transform. J. Math. Anal. Appl. 377
	(2011), 501--515.
	\bibitem{ghobber2} S. Ghobber, Ph. Jaming, Uncertainty principles for integral operators, arxiv:1206.1195v1
	
	
	\bibitem{HoughEtAl}J.B. Hough, M. Krishnapur, Y. Peres, B. Vir\'ag, Determinantal processes and independence.
	Probab. Surv. 3 (2006), 206--229.
	
	\bibitem{Hua} Hua Loo Keng,  Harmonic Analysis of Functions of Several Complex Variables in the Classical Domains,
	Science Press Peking 1958,
	Russian translation Moscow  Izd.Inostr. lit., 1959, English translation  (from the Russian) AMS 1963.
	
	\bibitem{Kallenberg}O. Kallenberg, Foundations of Modern Probability, 3rd ed., Springer, 2021.
	
	\bibitem{Khintchine}A. Ya. Khinchin, Mathematical methods of the theory of mass service, Trudy Mat. Inst. Steklov., 49, Acad. Sci. USSR, Moscow, 1955, 3–122 (in Russian).
	
	
	\bibitem{Kolmogorov}A. Kolmogoroff, Grundbegriffe der Wahrscheinlichkeitsrechnung, Springer Verlag, 1933.
	
	\bibitem{Lenard}A. Lenard, States of classical statistical mechanical systems of infinitely many particles. I.
	Arch. Rational Mech. Anal. 59 (1975), no. 3, 219--239.
	
	\bibitem{Lyons}R. Lyons, Determinantal probability measures.
	Publ. Math. Inst. Hautes \'Etudes Sci. No. 98 (2003), 167--212.
	
	\bibitem{LyonsSteif}R. Lyons, J. Steif, Stationary determinantal processes: phase multiplicity, Bernoullicity, entropy, and domination.
	Duke Math. J. 120 (2003), no. 3, 515--575.
	
	\bibitem{Lytvynov}E. Lytvynov, Fermion and boson random point processes as particle distributions of infinite free
	Fermi and Bose gases of finite density.
	Rev. Math. Phys. 14 (2002), no. 10, 1073--1098.
	
	\bibitem{Macchi}O. Macchi, The coincidence approach to stochastic point processes.
	Advances in Appl. Probability, 7 (1975), 83--122.
	
	\bibitem{Neretin} Yu.A. Neretin, Hua-type integrals over unitary groups and over projective limits of unitary groups.
	Duke Math. J. 114 (2002), no. 2, 239--266.
	
	\bibitem{Olsh0}G. Olshanski, The quasi-invariance property for the Gamma kernel determinantal measure.
	Adv. Math. 226 (2011), no. 3, 2305--2350.
	
	\bibitem{Olsh1}G. Olshanski, Unitary representations of infinite-dimensional pairs $(G,K)$ and the formalism of R. Howe,
	In: Representation of Lie Groups and Related Topics, Adv. Stud. Contemp. Math., 7, Gordon and Breach, NY, 1990, pp. 269--463,
	online at  \url{http://www.iitp.ru/upload/userpage/52/HoweForm.pdf}
	
	\bibitem{Olsh2}G. Olshanski, Unitary representations of infinite-dimensional classical groups
	(Russian), D.Sci Thesis, Institute of Geography of the Russian Academy of Sciences;
	online at \url{http://www.iitp.ru/upload/userpage/52/Olshanski_thesis.pdf}
	
	\bibitem{OlshVershik}G. Olshanski, A. Vershik, Ergodic unitarily invariant measures on the space of infinite Hermitian matrices.
	In: Contemporary mathematical physics, Amer. Math. Soc. Transl. Ser. 2, 175, Amer. Math. Soc., Providence, RI, 1996, pp.  137--175.
	
	\bibitem{Pickrell1}D. Pickrell, Mackey analysis of infinite classical motion groups.
	Pacific J. Math. 150 (1991), no. 1, 139--166.
	
	\bibitem{Pickrell2}D. Pickrell, Separable representations of automorphism groups of infinite symmetric spaces,
	J. Funct. Anal. 90 (1990), 1--26.
	
	\bibitem {Pickrell3}D. Pickrell, Measures on infinite-dimensional Grassmann manifolds,
	J. Funct. Anal. 70 (1987), 323--356.
	
	\bibitem{Rabaoui1}M. Rabaoui, Asymptotic harmonic analysis on the space of square complex matrices.
	J. Lie Theory 18 (2008), no. 3, 645--670.
	
	\bibitem{Rabaoui2}M. Rabaoui, A Bochner type theorem for inductive limits of Gelfand pairs.
	Ann. Inst. Fourier (Grenoble), 58 (2008), no. 5, 1551--1573.
	
	\bibitem{ReedSimon}M. Reed, B. Simon, Methods of modern mathematical physics. vol.I-IV.
	Second edition. Academic Press, Inc. New York, 1980.
	
	\bibitem{ShirTaka0}T. Shirai, Y. Takahashi, Random point fields associated with fermion, boson and other statistics.
	In: Stochastic analysis on large scale interacting systems, Adv. Stud. Pure Math., 39, Math. Soc. Japan, Tokyo, 2004, 345--354.
	
	\bibitem{ShirTaka1}T. Shirai, Y. Takahashi, Random point fields associated with certain Fredholm determinants.
	I. Fermion, Poisson and boson point processes. J. Funct. Anal. 205 (2003), no. 2, 414--463.
	
	\bibitem{ShirTaka2}T. Shirai, Y. Takahashi, Random point fields associated with certain Fredholm determinants.
	II. Fermion shifts and their ergodic and Gibbs properties. Ann. Probab. 31 (2003), no. 3, 1533--1564.
	
	\bibitem{Simon} B. Simon, Trace class ideals, AMS, 2011.
	
	\bibitem{Soshnikov}A. Soshnikov, Determinantal random point fields.
	(Russian) Uspekhi Mat. Nauk 55 (2000), no. 5(335), 107--160;
	translation in Russian Math. Surveys 55 (2000), no. 5, 923--975.
	
	\bibitem{Szego} G. Szeg{\"o}, Orthogonal polynomials, AMS 1969.
	
	\bibitem{TracyWidom}C. A. Tracy, H. Widom, Level spacing distributions and the Bessel kernel.
	Comm. Math. Phys. 161, no. 2 (1994), 289--309.
	
	\bibitem{Vershik}A.M. Vershik, A description of invariant measures for actions of certain infinite-dimensional groups.
	(Russian) Dokl. Akad. Nauk SSSR 218 (1974), 749--752.
	
	
\end{thebibliography}
\end{document}